\documentclass[a4paper]{article}

\usepackage{macros}

\title{Generalised BGP reflection functors via\\the Grothendieck construction}

\author{Tobias Dyckerhoff, Gustavo Jasso and Tashi Walde}

\date{}

\begin{document}

\maketitle

\begin{abstract}
  Inspired by work of Ladkani, we explain how to construct
generalisations of the classical reflection functors of Bern{\v{s}}te{\u\i}n,
Gel{\cprime}fand and Ponomarev by means of the Grothendieck construction.


\end{abstract}

\section*{Introduction}

Reflection functors were introduced to representation theory by
Bern{\v{s}}te{\u\i}n, Gel{\cprime}fand, and Ponomarev in the seminal article
\cite{BGP73} to provide a more transparent proof of Gabriel's Theorem
\cite{Gab72}. After fundamental work of Brenner and Butler \cite{BB76}, the BGP
reflection functors were extended by Auslander, Platzeck and Reiten
\cite{APR79} to Artin algebras with a simple projective (resp.\ injective)
module. It was then observed by Happel \cite{Hap87} that BGP reflection functors
can be understood conceptually in terms of derived equivalences induced by what
nowadays are called APR tilting complexes. These equivalences have been used,
for example, to show that the derived category of representations of a finite
quiver whose underlying graph is a tree does not depend on its orientation.

Generalising the BGP reflection functors for tree quivers, Ladkani \cite{Lad07}
constructed equivalences between derived categories of representations of finite
posets in an arbitrary abelian category. Ladkani uses these `generalised BGP
reflection functors' to establish derived equivalences between posets arising
naturally in representation theory and combinatorics \cite{Lad07a,Lad07b,Lad08}.
In a similar spirit, abstract versions of the BGP reflection functors were
developed by Groth and \Stovicek in a series of articles \cite{GS18b,GS16,GS16a}
with the most general form of their construction appearing as the main result in
\cite{GS18a}.

In this article we construct generalised BGP reflection functors by leveraging a
general e\-qui\-va\-lence between two stable $\infty$-categories associated to
an exact functor via gluing operations, thereby unifying the aforementioned
approaches. From the perspective of semi-orthogonal decompositions \cite{BK89},
these generalised BGP reflection functors can be interpreted as a mutation
between the two decompositions associated to an admissible subcategory of a
stable $\infty$-category. Note that the mentioned gluing operations cannot be
performed with triangulated categories and thus we are required to pass to a
richer framework such as that of Lurie's stable $\infty$-categories.

In \cref{sec:Grothendieck} we recall two---equivalent---procedures to
glue stable $\infty$-categories along an exact functor. In
\cref{sec:generalised} we construct the generalised BGP reflection
functors (\cref{thm:bipartiteGroth}) and discuss applications and examples.

\begin{acknowledgements}
  The authors thank Catharina Stroppel for detailed and valuable comments on an
  earlier version of this article. T.D.\ acknowledges the support by the 
VolkswagenStiftung for his Lichtenberg Professorship at the University of 
Hamburg; T.W.\ was supported by a Hausdorff Scholarship from the Bonn International 
Graduate School (BIGS) of Mathematics. 
\end{acknowledgements}

\section{Preliminaries: gluing along an exact functor}
\label{sec:Grothendieck}

Recall that a pointed $\infty$-category $\C$ with finite limits and finite
colimits is \introduce{stable} if the suspension functor
$\Sigma_\C\colon\C\to\C$ and the loop functor $\Omega_\C\colon\C\to\C$, given by
\[%
  \Sigma_\C\colon c\mapsto\cofib(c\to0)%
  \intxt{and}%
  \Omega_\C\colon c\mapsto\fib(0\to c),%
\]%
respectively, are mutually quasi-inverse equivalences, see Proposition~1.4.2.11
in \cite{Lur17}. The following construction is the \emph{raison d'\^{e}tre} for
the use of stable $\infty$-categories in this article.

\begin{construction}
  \label{const:Grothendieck}
  Let $F\colon\B\to\A$ be an exact functor between stable $\infty$-categories.
  \begin{enumerate}
  \item Let $\seccocaGroth{F}$ be the $\infty$-category defined by the following
    pullback diagram in the (very large) $\infty$-category $\biginfcat$ of
    $\infty$-categories:
    \begin{equation}
      \label{eq:pullback_seccocaGroth}
      \begin{tikzcd}
        \seccocaGroth{F}\rar\dar{\recproj}\pb&\Fun(\Delta^1,\A)\dar\\
        \B\rar[swap]{F}&\Fun(\set{0},\A)
      \end{tikzcd}
    \end{equation}
    Thus, informally, an object of $\seccocaGroth{F}$ can be identified with a
    triple
    \[
      (b,F(b)\to a)
    \]
    where $b$ is an object of $\B$, $a$ is object of $\A$, and $F(b)\to a$ is a
    morphism in $\A$. The $\infty$-category $\seccocaGroth{F}$ is stable as it
    is a limit of stable $\infty$-categories and exact functors\footnote{The
      $\infty$-category of stable $\infty$-categories is closed under limits of
      $\infty$-categories, see Theorem~1.1.4.4 in \cite{Lur17}.}.
  \item Dually, we define $\seccaGroth{F}$ to be the (stable) $\infty$-category
    defined by the following pullback diagram in $\biginfcat$:
    \begin{equation}
      \label{eq:pullback_secGroth}
      \begin{tikzcd}
        \seccaGroth{F}\rar\dar{\recproja}\pb&\Fun(\Delta^1,\A)\dar\\
        \B\rar[swap]{F}&\Fun(\set{1},\A)
      \end{tikzcd}
    \end{equation}
    Thus, informally, an object of $\seccaGroth{F}$ can be identified with a
    triple
    \[
      (b,a\to F(b))
    \]
    where $b$ is an object of $\B$, $a$ is an object of $\A$, and $a\to F(b)$ is
    a morphism in $\A$.
  \end{enumerate}
\end{construction}

\begin{remark}
  The stable $\infty$-category $\seccocaGroth{F}$ (resp.\ $\seccaGroth{F}$)
  associated to an exact functor $F$ can be identified\footnote{See for instance
    Lemma~5.4.7.15 in \cite{Lur09}.} with the stable $\infty$-category of
  sections of the contravariant (resp.\ covariant) Grothendieck construction.
  See Section~3.2 in \cite{Lur09} for more details on the Grothendieck
  construction.
\end{remark}

\begin{mylemma}
  \label{lemma:Grothendieck}
  Let $F\colon\B\to\A$ be an exact functor between stable $\infty$-categories.
  There are mutually quasi-inverse equivalences
  \begin{equation*}
    \Rplus\colon\seccaGroth{F}\lrlas\seccocaGroth{F}\noloc\Rminus
  \end{equation*}
  which, informally, are given by
  \begin{align*}
    \Rplus\colon&%
                  \left(b,\varphi\colon a\to F(b)\right)%
                  \lmapsto\left(b,F(b)\to\cofib(\varphi)\right),%
    \\\Rminus\colon&%
                     \left(b,\psi\colon F(b)\to a\right)%
                     \lmapsto\left(b,\fib(\psi)\to F(b)\right).%
  \end{align*}
\end{mylemma}
\begin{proof}
  Consider the solid commutative diagram of stable $\infty$-categories and exact
  functors
  \[%
    \begin{tikzcd}[column sep=0, row sep=small]%
      &\B%
      \ar[dd,equals]%
      \ar[rr,"F" near start]%
      &&\Fun(\set{1},\A)%
      \ar[dd,equals]\\%
      \seccaGroth{F}%
      \ar[ur]%
      \ar[dd,dashed,shift right=2,near end,"\Rplus"',crossing over]%
      \ar[crossing over,rr]%
      &&\Fun(\set{0\to1},\A)%
      \ar[ur]&\\%
      &\B%
      \ar[rr,"F" near start]%
      &&\Fun(\set{1},\A)\\%
      \seccocaGroth{F}%
      \ar[rr]%
      \ar[uu,dashed,shift right=2, near end,"\Rminus"',crossing over]%
      \ar{ur}%
      &%
      \phantom{\Fun(\set{0\to1},\A)}%
      &\Fun(\set{1\to2},\A)%
      \ar[ur]%
      \ar[from=uu,shift right=2,near end,"\Cofib"',crossing over]%
      \ar[uu,shift right=2,near end,"\Fib"',crossing over]&%
    \end{tikzcd}%
  \]%
  where the bottom and top squares are the pullback squares
  \eqref{eq:pullback_seccocaGroth} and \eqref{eq:pullback_secGroth},
  respectively. Since $\Cofib$ and $\Fib$ are mutually quasi-inverse
  equivalences, the functoriality of pullbacks implies the existence of the
  desired mutually quasi-inverse equivalences $\Rplus$ and $\Rminus$ which
  render the above cube commutative.
\end{proof}

The statement of \cref{lemma:Grothendieck} has the following interpretations
in terms of the classical concepts of recollements and semi-orthogonal
decompositions.

\begin{remark}
  Recollements were introduced by Be\u{\i}linson, Bern{\v{s}}te{\u\i}n and
  Deligne in \cite{BBD82} in the language of triangulated categories. It has
  since been observed that, when working with enhanced triangulated categories,
  recollements can be recovered from their gluing functors. For example, using
  differential graded categories as enhancements---as proposed in
  \cite{BK90}---the relevant theory is developed systematically in \cite{KL15};
  a treatment in the language of stable $\infty$-categories can be found in
  Appendix~A.8 in \cite{Lur17}. Recall that a recollement is a diagram of stable
  $\infty$-categories and exact functors of the form
  \begin{equation}
    \label{eq:recollement}
    \tikzrecollementdec \A\C\B \recinc\recproj
  \end{equation}
  such that the following conditions are satisfied:
  \begin{enumerate}
  \item The functor $\recinc$ is fully faithful and its essential image is
    precisely the kernel of $\recproj$. Moreover, there are adjunctions
    $\recincL\ladjto\recinc\ladjto\recincR$.
  \item There are adjunctions $\recprojL\ladjto\recproj\ladjto\recprojR$ and the
    functors $\recprojL$ and $\recprojR$ are fully faithful.
  \end{enumerate}
  The \emph{gluing functor} of a recollement of the form \eqref{eq:recollement}
  is the exact functor $\recincR\circ \recprojL\colon\B\to\A$. The results of
  Appendix~A.8 in \cite{Lur17} imply that the forgetful functor which associates
  to a recollement of stable $\infty$-categories its gluing functor induces an
  equivalence of $\infty$-categories between the $\infty$-category of
  recollements and the $\infty$-category $\Fun(\Delta^1,\stinfcat)$ of exact
  functors between stable $\infty$-categories. \cref{const:Grothendieck}
  provides two possible quasi-inverses to this equivalence while
  \cref{lemma:Grothendieck} provides a canonical identification between these
  quasi-inverses. Indeed, for an exact functor $F\colon\B\to\A$ between stable
  $\infty$-categories, the functors $\Rplus$ and $\Rminus$ induce mutually
  quasi-inverse equivalences of recollements
  \begin{equation*}
    \tikzrecollementmapdec\A{\seccaGroth{F}}\B{\recinc}\recproj
    \A{\seccocaGroth{F}}\B{\recinca}\recproja
    {\Sigma_\A}{\Rplus}{1_\B}
    \intxt{and}
    \tikzrecollementmapdec\A{\seccocaGroth{F}}\B{\recinca}\recproja
    \A{\seccaGroth{F}}\B{\recinc}\recproj
    {\Omega_\A}{\Rminus}{1_B}
  \end{equation*}
\end{remark}

\begin{remark}
  Semi-orthogonal decompositions were introduced by Bondal and Kapranov in \cite{BK89} in the
  language of triangulated categories. In their language, the datum of a
  recollement with middle term $\C$ is equivalent\footnote{See for instance
    Proposition~A.8.20 in \cite{Lur17}.} to the datum of the inclusion $\A
  \subset \C$ of an admissible subcategory, that is a subcategory such that the
  inclusion has left and right adjoints. To such a subcategory correspond two
  semi-orthogonal decompositions $\SOD{\A^{\perp}}{\A}$ and
  $\SOD{\A}{{^\perp}\A}$ of $\C$ which are \emph{mutations} of one another. In
  terms of the recollement data, the corresponding orthogonals are given by
  $\A^{\perp} = \recprojR(\B)$ and ${^\perp}\A = \recprojL(\B)$. Let
  $F\coloneqq\recincR\circ\recprojL\colon\B\to\A$ be the gluing functor.
  \begin{enumerate}
  \item The stable $\infty$-category $\seccocaGroth{F}$ is canonically
    equivalent to the $\infty$-category of arrows in $\C$ from $\A^{\perp}$ to
    $\A$ and hence, via the fibre functor, also equivalent to $\C$ itself.
  \item The stable $\infty$-category $\seccaGroth{F}$ is canonically equivalent
    to the $\infty$-category of arrows in $\C$ from $\A$ to $^{\perp}\A$ and
    hence, via the cofibre functor, also equivalent to $\C$ itself.
  \end{enumerate} Hence, the resulting equivalence
  $\seccocaGroth{F}\simeq\seccaGroth{F}$ (which agrees with the one from
  \cref{lemma:Grothendieck}) can be interpreted as passing from the
  description of the category $\C$ in terms of the semi-orthogonal decomposition
  $\SOD{\A^{\perp}}{\A}$ to a description in terms of the mutated decomposition
  $\SOD{\A}{{^\perp}\A}$. In particular, from this perspective, the generalised
  BGP reflection functors constructed below arise from mutations of
  semi-orthogonal decompositions.
\end{remark}

\section{Generalised BGP reflection functors}
\label{sec:generalised}

We fix a \emph{stable} $\infty$-category $\D$ throughout this section. For a
small $\infty$-category $Z$ we denote by $\diagramsin \D Z$ the
$\infty$-category $\Fun(Z,\D)$ of \introduce{$Z$-shaped diagrams in $\D$}. Our
aim is to construct equivalences of the form
\[
  \Rplus\colon%
  \diagramsin{\D}{Z}%
  \lrlas%
  \diagramsin{\D}{\refZ}%
  \noloc\Rminus%
\]
using the functors of \cref{lemma:Grothendieck}, where $\refZ$ is obtained
from $Z$ by \roughly{reflecting some arrows}. When $Z$ is a quiver or a poset
and $\D$ is the derived $\infty$-category of vector spaces over a field, these
equivalences reduce---after passing to homotopy categories---to triangle
equivalences between derived categories of representations. This is a
consequence of the following general fact, see for example Proposition~4.2.4.4
in \cite{Lur09}.

\begin{fact}
  Let $\D_\A$ be the derived $\infty$-category of a Grothendieck category $\A$.
  The derived category of the Grothendieck category $\A^Z$ of $Z$-shaped
  diagrams in $\A$ is equivalent to the homotopy category of the stable
  $\infty$-category $\diagramsin{\D_\A}{Z}$.
\end{fact}

\subsection{The main theorem}

Let $\baseQ=(\baseQ_0,\baseQ_1)$ be a quiver. For a $\baseQ$-shaped diagram
$\lf\colon\baseQ\to\infcat$ of small $\infty$-categories, we denote the
covariant (resp.\ contravariant) Grothendieck construction of $\lf$ by
$\gcocaGroth\lf\baseQ$ (resp. $\gcaGroth\lf\baseQ$); we refer the reader to
Section~3.2.5 in \cite{Lur09} for details\footnote{What we refer to as the
  Grothendieck construction is called the relative nerve in \cite{Lur09}; the
  explicit construction is given in Definition~3.2.5.2.}. Following Definition
4.12 in \cite{Lad08a}, we say that $\baseQ$ is \introduce{bipartite} if there
exists a decomposition $\baseQ_0=\baseQl\disjunion\baseQr$ such that all arrows
are of the form $\elbaseQl\to\elbaseQr$ for $\elbaseQl\in\baseQl$ and
$\elbaseQr\in\baseQr$. If $\baseQ$ is bipartite, then the (small)
$\infty$-categories $\gcocaGroth\lf\baseQ$ and $\gcaGroth\lf\baseQ$ are
characterised\footnote{Using the fact that $\baseQ$ is bipartite, these are
  direct applications of Theorem~7.4 and Corollary~7.6 in \cite{GHN17} (after
  unravelling the notation).\label{fn:Grothendieck_laxlimit}} up to equivalence
in terms of the pushout diagrams in the $\infty$-category $\infcat$ of small
$\infty$-categories
\begin{equation}
  \label{eq:GrothQF}
  \tikzpushouta[tiny]%
  {\coprod\limits_{\alpha\colon\elbaseQl\to\elbaseQr}\lf\elbaseQl}%
  {\coprod\limits_{\elbaseQr\in\baseQr}\lf\elbaseQr}%
  {%
    \coprod\limits_{\elbaseQl\in\baseQl}%
    {\lf\elbaseQl\times{\undercat\elbaseQl\baseQ}}%
  }%
  {\gcocaGroth{\lf}{\baseQ}}%
  \intxt{and}
  \tikzpushouta[tiny]%
  {\coprod\limits_{\alpha\colon\elbaseQl\to\elbaseQr}\lf\elbaseQl}%
  {\coprod\limits_{\elbaseQr\in\baseQr}\lf\elbaseQr}%
  {%
    \coprod\limits_{\elbaseQl\in\baseQl}%
    {\lf\elbaseQl\times{\overcat{\elbaseQl}{\baseQ^{\op}}}}%
  }%
  {\gcaGroth{\lf}{\baseQ}}%
\end{equation}
where $\undercat\elbaseQl\baseQ$ (resp.\ $\overcat{\elbaseQl}{\baseQ^{\op}}$)
denotes the slice category of objects under (resp.\ over) $\elbaseQl$. In the
case $\baseQ=\set{0\to 1}$, corresponding to a functor $f\colon X\to Y$ between
small $\infty$-categories, the $\infty$-categories $\gcocaGroth\lf\baseQ$ and
$\gcaGroth\lf\baseQ$ can be schematically illustrated as follows:
\begin{center}
  \includegraphics%
  [width=0.45\textwidth]%
  {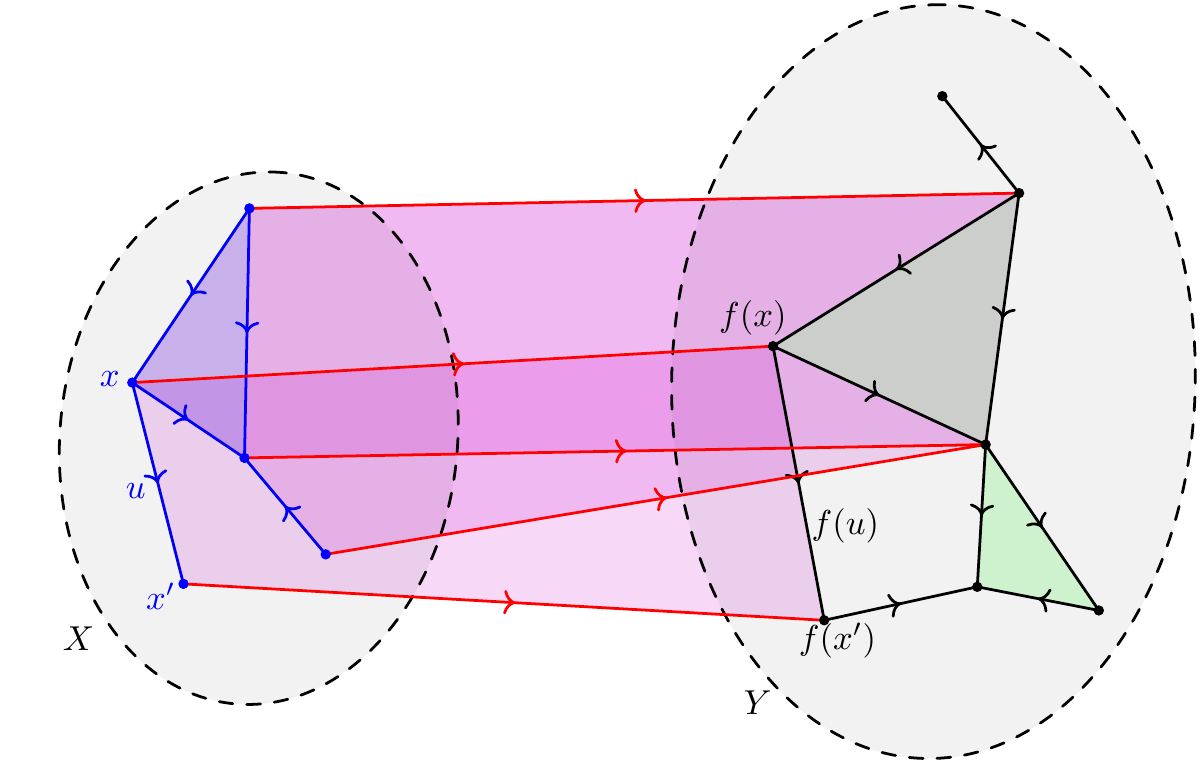}%
  \qquad\includegraphics%
  [width=0.45\textwidth]%
  {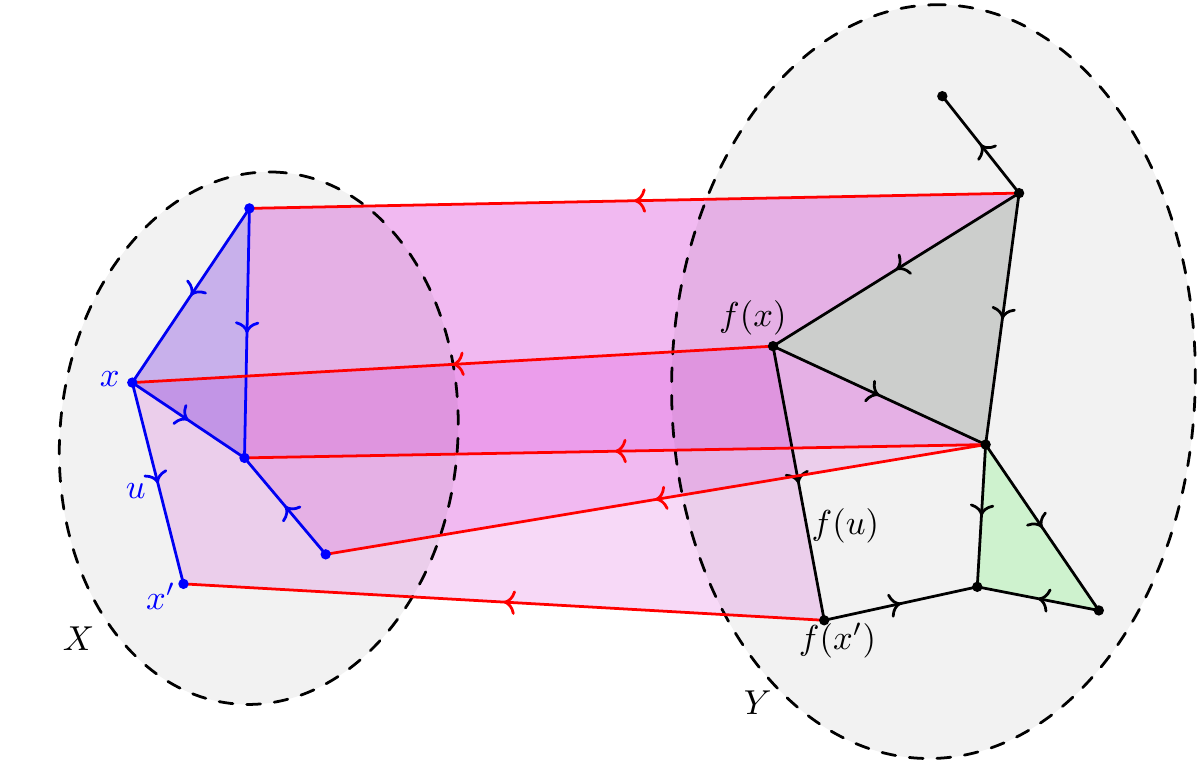}%
\end{center}
The coloured regions indicate commutativity relations.

\begin{remark}
  \label{rem:bipartitevsladkani}
  If the diagram $\lf$ takes values in \emph{ordinary} categories, then
  $\gcocaGroth\lf\baseQ$ and $\gcaGroth \lf\baseQ$ agree with (the nerve of) the
  corresponding $1$-categorical Grothendieck constructions. If $\lf$ furthermore
  takes values in posets, then $\gcocaGroth \lf\baseQ$ (resp.\ $\gcaGroth
  \lf\baseQ$) is again a poset if and only if, for every pair
  $\alpha,\beta\colon \elbaseQl\to \elbaseQr$ of parallel arrows in $\baseQ$ and
  every $x\in \lf\elbaseQl$, the elements $(\lf\alpha)(x)$ and $(\lf{\beta})(x)$
  of $\lf\elbaseQr$ have no common upper (resp.\ lower) bound. Under these
  assumptions---which are precisely the assumptions in the main theorem in
  \cite{Lad07}---the Grothendieck construction specialises to the construction
  therein.
\end{remark}

\begin{construction}
  \label{const:F}
  Let $\baseQ$ be a finite bipartite quiver with
  $\baseQ_0=\baseQl\disjunion\baseQr$ and $\lf\colon \baseQ\to \infcat$ a
  $\baseQ$-shaped diagram of small $\infty$-categories. We define
  $\infty$-categories
  \[
    X\coloneqq \coprod_{\elbaseQl\in \baseQl} \lf \elbaseQl%
    \intxt{and}%
    Y\coloneqq \coprod_{\elbaseQr\in \baseQr} \lf\elbaseQr%
  \]
  and an exact functor $F=F_{\lf}\colon\diagramsin{\D}{Y}\to\diagramsin{\D}{X}$
  as the composite
  \[
    \diagramsin{\D}{Y}\xra{\simeq}\bigoplus_{\elbaseQr\in
      \baseQr}\diagramsin{\D}{\lf\elbaseQr} \xra{(-\circ\lf\alpha)_\alpha}
    \bigoplus\limits_{\alpha\colon \elbaseQl\to\elbaseQr}
    \diagramsin{\D}{\lf\elbaseQl} \xra{\oplus} \bigoplus_{\elbaseQl\in
      \baseQl}\diagramsin{\D}{\lf\elbaseQl}\xra{\simeq}\diagramsin{\D}{X}.
  \]
  In particular, for $M\in\bigoplus_{\elbaseQr\in \baseQr}
  \diagramsin{\D}{\lf\elbaseQr}$ and $x\in\lf\elbaseQl$ we have
  \[
    (FM)_{x} =\bigoplus_{\alpha\colon\elbaseQl\to\elbaseQr}M_{(\lf\alpha)(x)}.
  \]
\end{construction}

The following result extends Ladkani's main theorem in \cite{Lad07} from posets
to small $\infty$-categories.

\begin{theorem}
  \label{thm:bipartiteGroth}
  In the setting of \cref{const:F}, there are canonical equivalences of
  stable $\infty$-categories
  \begin{equation}
    \label{eq:thmLF}    
    \seccaGroth{F}\simeq \diagramsin{\D}{\gcocaGroth \lf\baseQ}
    \intxt{and} \seccocaGroth{F}\simeq\diagramsin{\D}{\gcaGroth
      \lf\baseQ}.
  \end{equation}
  In particular, the functors of \cref{lemma:Grothendieck} induce mutually
  quasi-inverse equivalences of stable $\infty$-categories
  \begin{equation}
    \label{eq:Rplus_thm}
    \Rplus\colon%
    \diagramsin{\D}{\gcocaGroth \lf\baseQ}%
    \lrlas%
    \diagramsin{\D}{\gcaGroth \lf\baseQ}%
    \noloc\Rminus.%
  \end{equation}
  We call these equivalences \introduce{generalised BGP reflection functors}.
\end{theorem}

\begin{proof}
  We only establish the existence of the leftmost equivalence in
  \eqref{eq:thmLF}, the other one being analogous. Note that the existence of
  the desired equivalences \eqref{eq:Rplus_thm} then follows immediately from
  \cref{lemma:Grothendieck}.

  Firstly, applying the functor $\diagramsin\D{-}=\Fun(-, \D)$---which
  takes\footnote{%
    Indeed, limits in $\biginfcat$ can be detected after applying the functors
    $\Map_{\biginfcat}(\Z,-)\colon\biginfcat\to\SPACES$ for $\Z\in\biginfcat$
    and the functor $\Map_{\biginfcat}(\Z,\diagramsin{\D}{-})$ is equivalent to
    the representable functor $\Map_{\biginfcat}(-,\diagramsin{\D}{\Z})$ which
    sends colimits to limits.} pushouts in $\infcat$ to pullbacks in
  $\biginfcat$---to the diagram \eqref{eq:GrothQF} we see that the leftmost
  square in the diagram
  \begin{equation}
    \label{eq:inproof_pullback_Groth}
    \begin{tikzcd}[row sep=small]
      \diagramsin\D{\gcocaGroth{\lf}{\baseQ}}\ar[r]\ar[d]\pb& \diagramsin\D%
      {%
        \coprod_{\elbaseQl\in\baseQl}%
        \lf\elbaseQl\times\undercat\elbaseQl\baseQ%
      }%
      \ar[r]\ar[d]& \diagramsin\D%
      {\coprod_{\elbaseQl\in\baseQl}\lf\elbaseQl\times \Delta^1}%
      \ar[d,shift right=3.5em]\simeq\Fun(\Delta^1,\diagramsin{\D}{X})\\%
      \diagramsin\D Y \ar[r]& \diagramsin\D%
      {\coprod_{\alpha\colon\elbaseQl\to\elbaseQr}%
        \lf\elbaseQl} \ar[r]&%
      \diagramsin\D{\coprod_{\elbaseQl\in\baseQl}\lf\elbaseQl\times\set{1}}%
      \simeq\Fun(\set{1},\diagramsin{\D}{X})%
    \end{tikzcd}
  \end{equation}
  is a pullback of $\infty$-categories, where the bottom horizontal composite is
  precisely the functor $F$. Secondly the rightmost square can be identified
  with the pullback square
  \[
    \begin{tikzcd}[row sep=small]
      \bigoplus_{\elbaseQl\in\baseQl}%
      \Fun(\undercat\elbaseQl\baseQ,\diagramsin{\D}{\lf\elbaseQl})%
      \rar\dar\pb%
      &\bigoplus_{\elbaseQl\in\baseQl}%
      \Fun(\Delta^1,\diagramsin{\D}{\lf\elbaseQl})%
      \dar%
      \\\bigoplus_{\elbaseQl\in\baseQl}%
      \bigoplus_{\alpha\colon\elbaseQl\to\elbaseQr}%
      \diagramsin{\D}{\lf\elbaseQl}%
      \rar[swap]{\oplus}%
      &\bigoplus_{\elbaseQl\in\baseQl}%
      \Fun(\set{1},\diagramsin{\D}{\lf\elbaseQl)}%
    \end{tikzcd}
  \]
  Therefore the outer rectangle in~\eqref{eq:inproof_pullback_Groth} is also a
  pullback square. Finally, the claim follows by comparison with the pullback
  square \eqref{eq:pullback_secGroth} defining $\seccaGroth{F}$.
\end{proof}

\subsection{Applications}

We now explain how to recover the main result in \cite{GS18a} (see
\cref{coro:GS} below) as well as further results from \cite{Lad07} (see
\cref{coro:Delta1} and \cref{coro:conereflection} below) as special cases
of \cref{thm:bipartiteGroth}.

We begin by highlighting the most important instance of
\cref{thm:bipartiteGroth}. Let $\KronQ d$ be the $d$-Kronecker quiver, that
is the category with two objects $0,1$ and $d$ parallel arrows $0\to 1$. Let
$f_1,\dots, f_d\colon X\to Y$ be functors between small $\infty$-categories and
denote by $\dcocaGroth fd$ and $\dcaGroth fd$ the covariant and contravariant
Grothendieck constructions of the corresponding $\KronQ d$-shaped diagram of
$\infty$-categories.

\begin{mycorollary}
  \label{cor:KroneckerGroth}
  There are mutually quasi-inverse equivalences of stable $\infty$-categories
  \[
    \Rplus\colon%
    \diagramsin{\D}{\dcocaGroth fd}%
    \lrlas%
    \diagramsin{\D}{\dcaGroth fd}%
    \noloc\Rminus%
  \]
  induced by the functors of \cref{lemma:Grothendieck}.
\end{mycorollary}
\begin{proof}
  Apply \cref{thm:bipartiteGroth} to the (bipartite) quiver $\KronQ d$.
\end{proof}

\begin{example}
  Let $X=\set{0\to 1}$ and $Y=\set{0\to 1\to 2}$ and consider the functors
  $f_1,f_2\colon X\to Y$ given by $f_1(0)=0$ and $f_1(1)=1$, and
  $f_2(0)=f_2(1)=2$.
  The covariant and contravariant Grothendieck constructions
  $\dcocaGroth f2$ and $\dcaGroth f2$ of the resulting diagram
  $f\colon\KronQ{2}\to\cat$ are as follows:
  \begin{center}
    \includegraphics[width=0.33\textwidth]{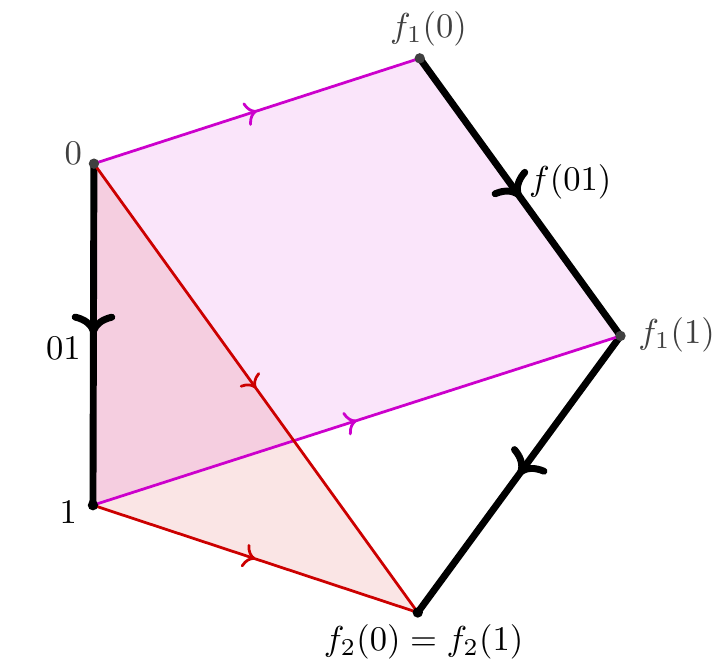}\qquad\qquad
    \includegraphics[width=0.33\textwidth]{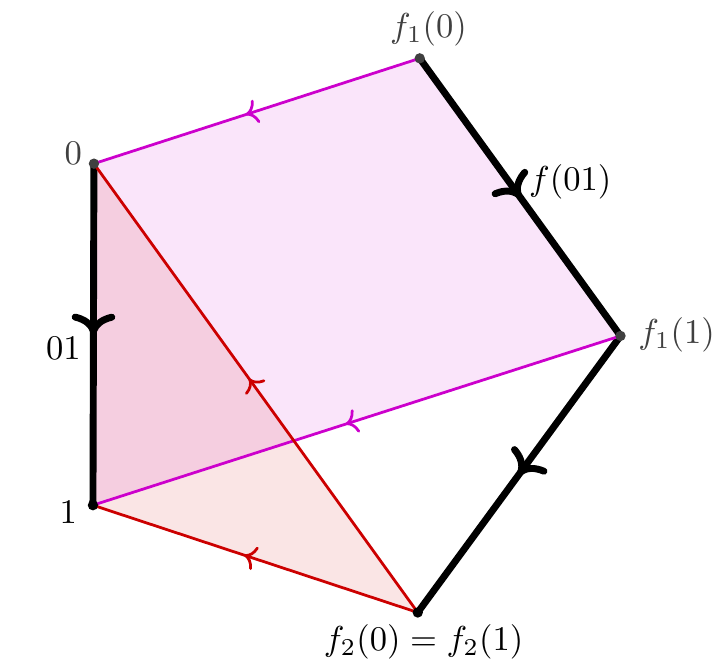}
  \end{center}
  The coloured regions indicate commutativity relations.
\end{example}

The following result recovers the classical BGP reflection functors by letting
$\D$ be the derived $\infty$-category of vector spaces over a field.

\begin{mycorollary}[Classical BGP reflection functors]
  \label{coro:GS}
  Let $Y$ be a small $\infty$-category and choose finitely many objects
  $y_1,\dots, y_d\in Y$ classified by functors $f_1,\dots,f_d\colon\pt\to Y$.
  There are mutually quasi-inverse equivalences
  \[
    \Rplus\colon%
    \diagramsin{\D}{\dcocaGroth fd}%
    \lrlas%
    \diagramsin{\D}{\dcaGroth fd}%
    \noloc\Rminus%
  \]
  induced by the functors of \cref{lemma:Grothendieck}, where
  $f\colon\KronQ{d}\to\infcat$ is the associated $\KronQ{d}$-shaped diagram.
\end{mycorollary}
\begin{proof}
  This is the case of~\cref{cor:KroneckerGroth} where $X$ is a point.
\end{proof}

\begin{remark}
  In the setting of \cref{coro:GS}, the $\infty$-categories $\dcocaGroth fd$
  and $\dcaGroth fd$ are obtained from $Y$ by adjoining a new source (resp.\
  sink) $y$ with \emph{free} arrows $y\to y_i$ (resp.\ $y_i\to y$). In
  particular, if $Y$ is a poset, then $\dcocaGroth fd$ and $\dcaGroth fd$ are
  almost never posets. The functor $\Rplus$ acts as follows: Given a
  representation
  \[
    M\colon\dcocaGroth fd\lra\D
  \]
  the representation
  \[
    \Rplus M\colon\dcaGroth fd\lra\D
  \]
  is given by $(\Rplus M)_{y'}=M_{y'}$ if $y'\neq y$ and
  \[
    (\Rplus M)_y=\cofib\left(M_y\to\bigoplus_i M_{y_i}\right).
  \]
  The action of $\Rminus$ can be described similarly, in terms of the fibre
  functor. This description is in complete analogy with the classical BGP
  reflection functors; moreover, one can show that it agrees with the abstract
  reflection functors of \cite{GS18a}.
\end{remark}

The following result extends Corollary~1.3 in \cite{Lad07} from posets to small
$\infty$-categories.

\begin{mycorollary}
  \label{coro:Delta1}
  Let $f\colon X\to Y$ be a functor between small $\infty$-categories. There are
  mutually quasi-inverse equivalences
  \[
    \Rplus\colon%
    \diagramsin{\D}{\cocaGroth{f}}%
    \lrlas%
    \diagramsin{\D}{\caGroth{f}}%
    \noloc\Rminus%
  \]
  induced by the functors of \cref{lemma:Grothendieck}.
\end{mycorollary}
\begin{proof}
  This is the case $d=1$ of~\cref{cor:KroneckerGroth}.
\end{proof}

\begin{example}
  Let $X=\set{0\to 1\to 2}$ and $Y=\set{0\to 1}$ and consider the functor
  $f\colon X\to Y$ given by $f(0)=f(1)=0$ and $f(2)=1$. The covariant and
  contravariant Grothendieck constructions $\cocaGroth f$ and $\caGroth f$ of
  the induced diagram $f\colon\Delta^1\to\cat$ are as follows:
  \begin{center}
    \includegraphics[width=0.33\textwidth]{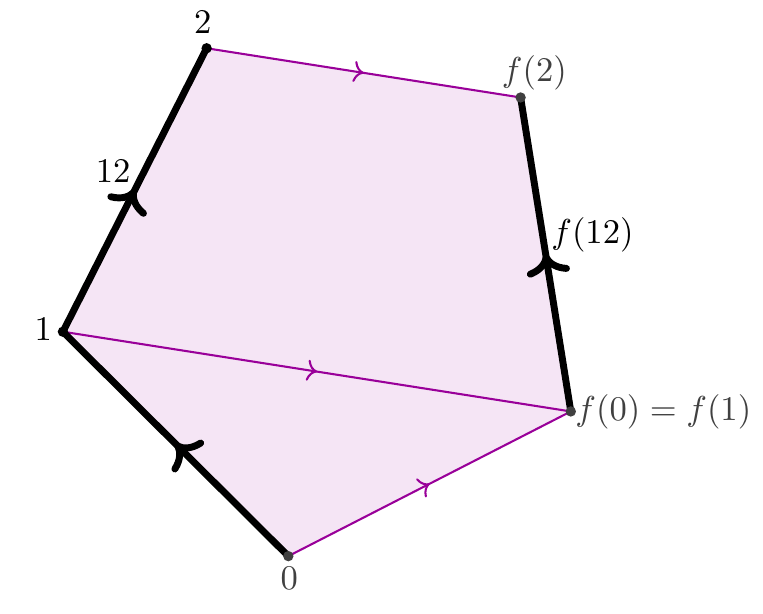}\qquad\qquad
    \includegraphics[width=0.33\textwidth]{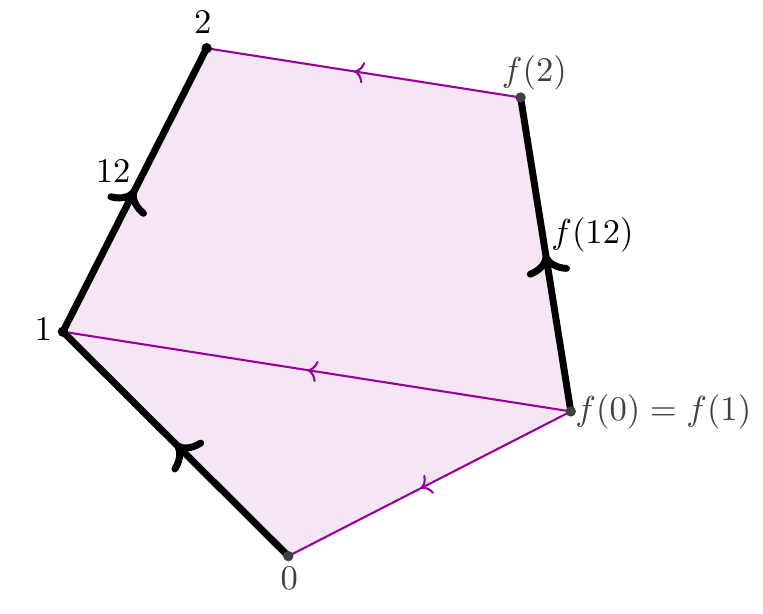}
  \end{center}
  The coloured regions indicate commutativity relations.
\end{example}

The following result extends Corollary~1.5 in \cite{Lad07} from posets to small
$\infty$-categories. Recall that, given an $\infty$-category $X$, there are
$\infty$-categories $X^{\triangleright}$ and $X^{\triangleleft}$ obtained by
adding to $X$ a terminal object $\conept$ or an initial object $\mconept$,
respectively.

\begin{mycorollary}
  \label{coro:conereflection}
  Let $X$ be a small $\infty$-category. There are mutually quasi-inverse
  equivalences
  \[
    \Rplus\colon%
    \diagramsin{\D}{X^{\triangleright}}%
    \lrlas%
    \diagramsin{\D}{X^{\triangleleft}}%
    \noloc\Rminus%
  \]
  induced by the functors of \cref{lemma:Grothendieck}.
\end{mycorollary}
\begin{proof}
  Apply \cref{coro:Delta1} to the unique functor $f\colon X\to\pt$ and note
  that $\cocaGroth{f}$ and $\caGroth{f}$ are equivalent to $X^{\triangleright}$
  and $X^{\triangleleft}$, respectively.
\end{proof}

\begin{remark}
  In the situation of \cref{coro:conereflection} the functor $\Rplus$ acts as
  follows: For a representation
  \[
    M\colon X^{\triangleright}\lra\D,
  \]
  the representation
  \[
    \Rplus M\colon X^{\triangleleft}\lra\D
  \]
  is given by $(\Rplus M)_{\mconept}=M_{\conept}$ and
  \[
    (\Rplus M)_x=\cofib\left(M_x\to M_{\conept}\right)
  \]
  for each object $x$ of $X$. The action of $\Rminus$ can be described
  similarly, in terms of the fibre functor. Note the stark contrast with the
  action of the reflection functors of \cref{coro:GS}, which deals with the
  case of freely adjoined sinks or sources.
\end{remark}


\bibliographystyle{amsalpha}
\bibliography{library}

\vspace{2em}

\noindent%
(Dyckerhoff) %
Universit\"{a}t Hamburg, Fachbereich Mathematik, Bundesstra{\ss}e 55, 20146
Hamburg, Germany. E-mail address: \href{mailto:tobias.dyckerhoff@uni-hamburg.de}{\texttt{tobias.dyckerhoff@uni-hamburg.de}}%

\medskip

\noindent%
(Jasso) %
Rheinische Friedrich-Wilhelms-Universit\"{a}t Bonn, Mathematisches Institut,
Endenicher Allee 60, 53115 Bonn, Germany. E-mail address: \href{mailto:gjasso@math.uni-bonn.de}{\texttt{gjasso@math.uni-bonn.de}}%

\medskip    

\noindent%
(Walde) %
Rheinische Friedrich-Wilhelms-Universit\"{a}t Bonn, Mathematisches Institut,
Endenicher Allee 60, 53115 Bonn, Germany. E-mail address: \href{mailto:twalde@math.uni-bonn.de}{\texttt{twalde@math.uni-bonn.de}}%

\end{document}